\documentclass[a4paper,reqno,11pt]{amsart}

\usepackage{fullpage}
\usepackage{microtype}
\usepackage{mathtools}
\usepackage{amssymb}

\usepackage{enumitem}
\setlist[enumerate]{label=(\roman*), font=\normalfont}

\usepackage[%
    bookmarks=true,%
    bookmarksnumbered=true,%
    colorlinks=true,%
    setpagesize=false,%
    linkcolor=blue,%
    citecolor=magenta,%
    filecolor=magenta,%
    urlcolor=magenta,%
]{hyperref}

\usepackage{aliascnt}
\usepackage{cleveref}

\def\cCrefname#1#2#3{%
    \crefname{#1}{#2}{#3}%
    \Crefname{#1}{#2}{#3}%
}

\cCrefname{section}{Section}{Sections}
\cCrefname{subsection}{Section}{Sections}
\cCrefname{thm}{Theorem}{Theorems}
\cCrefname{prop}{Proposition}{Propositions}
\cCrefname{lem}{Lemma}{Lemmas}
\cCrefname{rem}{Remark}{Remarks}

\newtheorem{thm}{Theorem}[section]

\newaliascnt{prop}{thm}
\newtheorem{prop}[prop]{Proposition}
\aliascntresetthe{prop}

\newaliascnt{lem}{thm}
\newtheorem{lem}[lem]{Lemma}
\aliascntresetthe{lem}

\theoremstyle{remark}

\newaliascnt{rem}{thm}
\newtheorem{rem}[rem]{Remark}
\aliascntresetthe{rem}

\newcommand{\calN}{\mathcal{N}}

\newcommand{\ZZ}{\mathbb{Z}}
\newcommand{\RR}{\mathbb{R}}
\newcommand{\kk}{\Bbbk}

\newcommand{\eb}{\mathbf{e}}
\newcommand{\one}{\mathbf{1}}

\newcommand{\set}[1]{\left\{ #1 \right\}}
\newcommand{\setcond}[2]{\set{#1 : #2}}
\newcommand{\rbra}[1]{\left( #1 \right)}

\DeclarePairedDelimiter{\card}{\lvert}{\rvert}

\DeclareMathOperator{\conv}{conv}
\DeclareMathOperator{\Fill}{Fill}
\let\int\relax
\DeclareMathOperator{\int}{int}

\begin{document}

\title{A palindromicity criterion for the $h$-polynomials of bipartite edge rings}
\author{Yuta Hatasa}

\email{hatasa.y.b8ad@gmail.com}


\begin{abstract}
	We study a symmetry problem for the $h$-polynomials of edge rings of bipartite graphs.
	Let $G$ be a bipartite graph and write
	$h(\kk[G];t)=h_0+h_1t+\cdots+h_st^s$.
	We prove that if $\kk[G]$ is pseudo-Gorenstein and $h_1=h_{s-1}$, then $\kk[G]$ is Gorenstein.
	Equivalently, under these assumptions the $h$-polynomial of $\kk[G]$ is palindromic.
	The proof treats the $2$-connected case first by translating the numerical condition $h_1=h_{s-1}$ into a tight-separation condition for non-edges, and then passes to arbitrary bipartite graphs using the block decomposition.
	We also construct a blockwise minimal Gorenstein closure, obtained by adjoining all non-edges not separated by tight acceptable sets, and show that this construction preserves the next-to-leading coefficient of the $h$-polynomial.
\end{abstract}

\maketitle

\section{Introduction}

Let $R=\bigoplus_{k\ge 0}R_k$ be a Cohen--Macaulay homogeneous domain of dimension $d$ over a field $R_0=\kk$.
The Hilbert series of $R$ has the form
\[
	\sum_{k\ge 0}(\dim_\kk R_k)t^k
	=
	\frac{h_0+h_1t+\cdots+h_st^s}{(1-t)^d},
	\quad h_s\ne 0.
\]
The numerator $h(R;t)=h_0+h_1t+\cdots+h_st^s$ is called the \textit{$h$-polynomial} of $R$.
Since $R_0=\kk$, we have $h_0=1$, and Cohen--Macaulayness implies that all $h_i$ are non-negative integers.
The ring $R$ is called \textit{pseudo-Gorenstein} if $h_s=1$.
When $R$ is a Cohen--Macaulay homogeneous domain, a theorem of Stanley states that $R$ is Gorenstein if and only if the $h$-vector is symmetric, that is, $h_i=h_{s-i}$ for all $0\le i\le s$~\cite{stanley1978hilbert}.
Equivalently, $h(R;t)$ is \textit{palindromic}, meaning that $h(R;t)=t^s h(R;t^{-1})$.

The $h$-polynomial encodes the Hilbert function and relates numerical invariants of a graded ring to its algebraic properties.
For example, $h(R;1)$ is the multiplicity of $R$, while symmetry characterizes Gorensteinness in the setting above.
Relations between $h$-vectors and almost Gorensteinness have also been studied by Higashitani~\cite{higashitani2016almost}.
Recent work on edge rings includes explicit computations of $h$-vectors by Higashitani and Shibu Deepthi~\cite{higashitani2023h}.
Bhaskara, Higashitani, and Shibu Deepthi~\cite{bhaskara2026h} extended these computations to odd-cycle compositions and characterized the almost Gorenstein edge rings in this family.
Leading coefficients and further almost Gorenstein families were studied by Hatasa, Kowaki, and Matsushita~\cite{hatasa2025pseudo}.
These studies motivate the question of how much information about the coefficients suffices to determine Gorensteinness.

We study this question for edge rings of bipartite graphs.
For a finite simple graph $G$ on the vertex set $V(G)=[d]$, let $E(G)$ denote its edge set.
The \textit{edge ring} of $G$ is
\[
	\kk[G]=\kk[t_it_j:\{i,j\}\in E(G)]\subset \kk[t_1,\ldots,t_d].
\]
The edge rings of bipartite graphs are normal by results of Ohsugi--Hibi and Simis--Vasconcelos--Villarreal~\cite{ohsugi1998normal,simis1998integral}, and hence Cohen--Macaulay by a theorem of Hochster~\cite{hochster1972rings}.

For bipartite graphs, Hatasa, Kowaki, and Matsushita proved that $\kk[G]$ is pseudo-Gorenstein if and only if every block of $G$ is matching-covered~\cite{hatasa2025pseudo}.
In particular, if $G$ is $2$-connected and bipartite, the pseudo-Gorenstein property is equivalent to $G$ being matching-covered.
One of the next natural questions is whether a small amount of additional symmetry of the $h$-vector already forces full symmetry.
Our main result gives an affirmative answer for bipartite graphs.

\begin{thm}\label{thm:main}
	Let $G$ be a bipartite graph and write
	\[
		h(\kk[G];t)=h_0+h_1t+\cdots+h_st^s.
	\]
	If $h_s=1$ and $h_1=h_{s-1}$, then $\kk[G]$ is Gorenstein.
	In particular, $h(\kk[G];t)$ is palindromic.
\end{thm}

The same method also gives a blockwise Gorenstein closure for pseudo-Gorenstein bipartite edge rings.

\begin{thm}\label{thm:blockwise-closure}
	Let $G$ be a bipartite graph and write
	\[
		h(\kk[G];t)=h_0+h_1t+\cdots+h_st^s.
	\]
	Suppose that $h_s=1$.
	For each block $B$ of $G$, choose a bipartition $X_B\sqcup Y_B$.
	Let $\Fill(B)$ be the set of non-edges $\{x,y\}$ of $B$, with $x\in X_B$ and $y\in Y_B$, that are not separated by any tight acceptable set $T\subset X_B$, in the sense that no such $T$ satisfies
	\[
		x\in T,\qquad y\notin N_B(T).
	\]
	For each block $B$, let $\widehat{B}$ be the graph on the same vertex set as $B$ with
	\[
		E(\widehat{B})=E(B)\cup\Fill(B).
	\]
	Let $\widehat{G}$ be the graph obtained from $G$ by replacing each block $B$ by $\widehat{B}$.
	Then the following assertions hold.
	\begin{enumerate}
		\item $\kk[\widehat{G}]$ is Gorenstein.
		\item Writing
		      \[
			      h(\kk[\widehat{G}];t)=\widehat{h}_0+\widehat{h}_1t+\cdots+\widehat{h}_{\widehat{s}}t^{\widehat{s}},
		      \]
		      one has $\widehat{s}=s$ and $\widehat{h}_{s-1}=h_{s-1}$.
		\item The construction is blockwise minimal: if $B$ is a block of $G$ with at least one edge and $K_B$ is a bipartite graph with bipartition $X_B\sqcup Y_B$ such that $E(B)\subset E(K_B)$ and $\kk[K_B]$ is Gorenstein, then
		      \[
			      E(\widehat{B})\subset E(K_B).
		      \]
	\end{enumerate}
\end{thm}

The proof first handles the $2$-connected case.
In that case, we express the difference $h_{s-1}-h_1$ as the number of specific interior lattice points of the edge polytope arising from non-edges of $G$.
Thus $h_1=h_{s-1}$ is equivalent to the assertion that every non-edge is separated by a tight acceptable set on the $X$-side, and symmetrically on the $Y$-side.
We then prove a purely graph-theoretic statement: in a $2$-connected matching-covered bipartite graph, this $X$-side separation condition is equivalent to the tightness of every acceptable set.
By the Ohsugi--Hibi criterion, this is exactly the Gorenstein condition.
Finally, for a bipartite graph, the edge ring decomposes as a tensor product over its blocks, so its $h$-polynomial is the product of the $h$-polynomials of the blocks.
This reduces the general bipartite case to the $2$-connected one.

\section{Preliminaries}

\subsection{Edge polytopes and facets}
\label{subsec:edge-polytopes}

Throughout this paper, all graphs are finite and simple.
Let $G$ be a graph on the vertex set $V(G)=[d]$.
For an edge $e=\{i,j\}$, set $\rho(e)=\eb_i+\eb_j\in\RR^d$.
The \textit{edge polytope} of $G$ is
\[
	P_G=\conv\{\rho(e):e\in E(G)\}\subset\RR^d.
\]
If $G$ is connected and bipartite, then $\dim P_G=d-2$, and $\dim\kk[G]=d-1$; see~\cite[Chapter~10]{villarreal2001monomial}.
Moreover, since $\kk[G]$ is normal, the Hilbert function of $\kk[G]$ is the Ehrhart function of $P_G$~\cite[Chapters~9--10]{villarreal2001monomial}:
\[
	\dim_\kk \kk[G]_k=\card{kP_G\cap\ZZ^d}.
\]

Let $G$ be bipartite with bipartition $V(G)=X\sqcup Y$.
For $T\subset X$, set
\[
	N_G(T)=\setcond{y\in Y}{\{x,y\}\in E(G)\text{ for some }x\in T},
	\qquad
	\delta_G(T)=\card{N_G(T)}-\card{T}.
\]
We use the same notation for subsets of $Y$.
In this paper, we call a subset $T$ \textit{tight} if $\delta_G(T)=1$.
Following Ohsugi--Hibi~\cite{ohsugi1998normal}, a non-empty set $T\subset X$ is called \textit{acceptable} if the bipartite graph $B(T)$ induced by the edges between $T$ and $N_G(T)$ is connected and the induced graph on
\[
	(X\setminus T)\sqcup (Y\setminus N_G(T))
\]
is connected with at least one edge.
Acceptable subsets of $Y$ are defined in the same way.

A vertex $v$ of a connected graph is called \textit{ordinary} if $G\setminus v$, the induced graph obtained by deleting $v$, is connected.
With this terminology, a connected graph is called \textit{$2$-connected} if all of its vertices are ordinary.
The facet description of Ohsugi--Hibi states that, for a connected bipartite graph, the coordinate facets of $P_G$ are precisely the facets supported by
\[
	z_v=0
\]
for ordinary vertices $v$.
The remaining facets are supported by the hyperplanes
\[
	\sum_{y\in N_G(T)}z_y-\sum_{x\in T}z_x=0
\]
for acceptable sets $T\subset X$; equivalently, one may use the symmetric formulation with acceptable subsets of $Y$.
In particular, if $G$ is $2$-connected, then every vertex is ordinary, so every coordinate hyperplane $z_v=0$ supports a facet of $P_G$.

\subsection{Blocks}

A \textit{block} of $G$ is a maximal $2$-connected subgraph.

We use the following decomposition of bipartite edge rings along blocks.

\begin{prop}[{\cite[Proposition~10.1.48]{villarreal2001monomial}}]\label{prop:block-product}
	Let $G$ be a bipartite graph, and let $B_1,\ldots,B_m$ be the blocks of $G$.
	Then
	\[
		\kk[G]\cong \kk[B_1]\otimes_\kk\cdots\otimes_\kk\kk[B_m]
	\]
	as standard graded $\kk$-algebras.
	In particular,
	\[
		h(\kk[G];t)=\prod_{i=1}^m h(\kk[B_i];t).
	\]
\end{prop}

\subsection{Matching-covered bipartite graphs}

A connected graph $G$ is \textit{matching-covered} if every edge of $G$ is contained in a perfect matching.
Recall that Hall's condition for a bipartite graph with bipartition $X\sqcup Y$ is
\[
	\card{W}\le\card{N_G(W)}\qquad\text{for every }W\subset X.
\]
Hall's theorem states that this condition is equivalent to the existence of a matching covering $X$; when $\card{X}=\card{Y}$, such a matching is perfect.
The proof of the main theorem uses the following elementary strict Hall property for the $2$-connected matching-covered blocks that occur there.

\begin{lem}\label{lem:strict-hall}
	Let $G$ be a $2$-connected matching-covered bipartite graph with bipartition $X\sqcup Y$.
	Then $\card{X}=\card{Y}$ and, for every non-empty proper subset $A\subsetneq X$,
	\[
		\card{N_G(A)}\ge \card{A}+1.
	\]
	The analogous statement holds for non-empty proper subsets of $Y$.
\end{lem}

\begin{proof}
	The equality $\card{X}=\card{Y}$ follows from the existence of a perfect matching.
	Hall's theorem gives $\card{N_G(A)}\ge \card{A}$.
	Suppose that $\card{N_G(A)}=\card{A}$ for some non-empty proper subset $A\subsetneq X$.
	Then every perfect matching matches all vertices of $A$ with the vertices of $N_G(A)$.
	Hence no edge from $X\setminus A$ to $N_G(A)$ can be contained in a perfect matching.
	Since $G$ is connected and $A$ is a non-empty proper subset, such an edge exists, contradicting the matching-covered property.
\end{proof}

\section{The coefficient condition}
\label{sec:coefficient}

We first record a general formula for the linear coefficient of the $h$-polynomial.

\begin{lem}\label{lem:h1}
	Let $G$ be a connected bipartite graph.
	Put $d=\card{V(G)}$ and write
	\[
		h(\kk[G];t)=h_0+h_1t+\cdots+h_st^s.
	\]
	Then
	\[
		h_1=\card{E(G)}-d+1.
	\]
\end{lem}

\begin{proof}
	Since $G$ is connected and bipartite, $\dim\kk[G]=d-1$.
	The lattice points of $P_G$ are precisely the points $\rho(e)$ for $e\in E(G)$~\cite[Chapter~5, ``Edge Polytopes and Edge Rings'']{herzog2018binomial}, so
	\[
		\card{P_G\cap\ZZ^d}=\card{E(G)}.
	\]
	Comparing the coefficients of degrees zero and one in the Ehrhart series
	\[
		\sum_{k\ge 0}\card{kP_G\cap\ZZ^d}t^k
		=
		\frac{h_0+h_1t+\cdots+h_st^s}{(1-t)^{d-1}}
	\]
	gives $h_0=1$ and $\card{E(G)}=(d-1)+h_1$.
\end{proof}

For the rest of this section, let $G$ be a $2$-connected matching-covered bipartite graph with bipartition $X\sqcup Y$.
By \Cref{lem:strict-hall}, $\card{X}=\card{Y}$; write $\card{X}=\card{Y}=n$ and
\[
	h(\kk[G];t)=h_0+h_1t+\cdots+h_st^s.
\]
Then $\dim P_G=2n-2$ and $\dim\kk[G]=2n-1$.
For every $k\ge1$, each point $z\in kP_G$ satisfies
\[
	\sum_{x\in X}z_x=\sum_{y\in Y}z_y=k,
\]
because each edge vector has one non-zero coordinate in each part of the bipartition.
Throughout the paper, $\int$ denotes relative interior in the affine span of the polytope.
Every lattice point in $\int(kP_G)$ has positive integral coordinates, since every coordinate hyperplane supports a facet by \Cref{subsec:edge-polytopes}.
Thus $k\ge n$, and for $k=n$ the only possible interior lattice point is
\[
	\one=(1,\ldots,1)\in\ZZ^{2n}.
\]
This point does lie in $\int(nP_G)$: it is the sum of the edge vectors of a perfect matching, so it belongs to $nP_G$, and it satisfies every facet inequality strictly.
Indeed, its coordinates are positive, and the linear form associated to an acceptable set $T$ takes the value $\delta_G(T)\ge1$ by \Cref{lem:strict-hall}.

Ehrhart reciprocity, in its generating-function form, gives
\begin{equation}\label{eq:interior-series}
	\sum_{k\ge1}\card{\int(kP_G)\cap\ZZ^{2n}}t^k
	=
	\frac{t^{2n-1}h(\kk[G];t^{-1})}{(1-t)^{2n-1}};
\end{equation}
see~\cite[Chapter~9]{villarreal2001monomial}.
The first non-zero term on the left is $t^n$, with coefficient one.
The first non-zero term on the right is $h_s t^{2n-1-s}$.
Hence $s=n-1$ and $h_s=1$, the latter also following from the matching-covered criterion for pseudo-Gorensteinness~\cite{hatasa2025pseudo}.
The remaining results in this section use these assumptions.

\begin{lem}\label{lem:hs-1}
	Let $G$ be a $2$-connected matching-covered bipartite graph with bipartition $X\sqcup Y$, and use the notation above.
	Then
	\[
		h_{s-1}=\card{\int((n+1)P_G)\cap\ZZ^{2n}}-2n+1.
	\]
\end{lem}

\begin{proof}
	Taking the coefficient of $t^{n+1}$ in \eqref{eq:interior-series} and using $s=n-1$ and $h_s=1$, we obtain
	\[
		\card{\int((n+1)P_G)\cap\ZZ^{2n}}
		=
		h_{s-1}+(2n-1)h_s
		=
		h_{s-1}+2n-1.\qedhere
	\]
\end{proof}

\begin{prop}\label{prop:difference}
	Let $G$ be a $2$-connected matching-covered bipartite graph with bipartition $X\sqcup Y$, and use the notation above.
	Let
	\[
		\calN_G=\setcond{(x,y)\in X\times Y}{\{x,y\}\notin E(G)}
	\]
	be the set of non-edges across the bipartition.
	Then
	\[
		h_{s-1}-h_1
		=
		\card{\setcond{(x,y)\in\calN_G}{\one+\eb_x+\eb_y\in\int((n+1)P_G)}}.
	\]
\end{prop}

\begin{proof}
	Every lattice point of $(n+1)P_G$ has non-negative integral coordinates, and the sums of its $X$-coordinates and $Y$-coordinates are both $n+1$.
	If such a point is in the interior, then all its coordinates are positive.
	Indeed, for each vertex $v$, the coordinate hyperplane $z_v=0$ supports a facet of $P_G$ by \Cref{subsec:edge-polytopes}, while $P_G$ is contained in the half-space $z_v\ge 0$.
	Thus an interior point cannot lie on any coordinate hyperplane.
	Hence every lattice point of $\int((n+1)P_G)$ has the form
	\[
		\one+\eb_x+\eb_y
	\]
	for some $x\in X$ and $y\in Y$.
	If $\{x,y\}\in E(G)$, then
	\[
		\one+\eb_x+\eb_y=\one+\rho(\{x,y\})
	\]
	lies in $\int((n+1)P_G)$, since $\one\in\int(nP_G)$.
	Therefore,
	\[
		\card{\int((n+1)P_G)\cap\ZZ^{2n}}
		=
		\card{E(G)}
		+
		\card{\setcond{(x,y)\in\calN_G}{\one+\eb_x+\eb_y\in\int((n+1)P_G)}}.
	\]
	The assertion follows from \Cref{lem:h1,lem:hs-1}.
\end{proof}

\begin{prop}\label{prop:nonedge-cut}
	Let $G$ be a $2$-connected matching-covered bipartite graph with bipartition $X\sqcup Y$, and use the notation above.
	The following conditions are equivalent:
	\begin{enumerate}
		\item $h_1=h_{s-1}$;
		\item for every non-edge $(x,y)\in\calN_G$, there exists a tight acceptable set $T\subset X$ such that
		      \[
			      x\in T,\quad y\notin N_G(T);
		      \]
		\item for every non-edge $(x,y)\in\calN_G$, there exists a tight acceptable set $S\subset Y$ such that
		      \[
			      y\in S,\quad x\notin N_G(S).
		      \]
	\end{enumerate}
\end{prop}

\begin{proof}
	We first prove that (i) implies (ii).
	By \Cref{prop:difference}, the equality $h_1=h_{s-1}$ means precisely that none of the points $\one+\eb_x+\eb_y$ corresponding to non-edges lies in $\int((n+1)P_G)$.
	Fix a non-edge $(x,y)\in\calN_G$ and put
	\[
		p=\one+\eb_x+\eb_y.
	\]
	The point $p$ belongs to the affine span of $(n+1)P_G$, since the sum of its $X$-coordinates and the sum of its $Y$-coordinates are both $n+1$.
	Moreover, $p$ has positive coordinates, so it satisfies all strict inequalities defined by the coordinate facets.
	Therefore, if $p$ is not in $\int((n+1)P_G)$, then, by the Ohsugi--Hibi facet description in the $X$-side form, $p$ must fail the strict inequality corresponding to an acceptable set $T\subset X$.
	We evaluate the linear form defining this facet at $p$.
	The vector $\one$ contributes $\card{N_G(T)}-\card{T}=\delta_G(T)$.
	Since $x\in X$ and $y\in Y$, the additional vector $\eb_y$ contributes $1$ exactly when $y\in N_G(T)$, while the additional vector $\eb_x$ is subtracted and contributes $-1$ exactly when $x\in T$.
	Equivalently, if
	\[
		\epsilon_y=
		\begin{cases}
			1 & \text{if } y\in N_G(T), \\
			0 & \text{otherwise}
		\end{cases}
		\quad\text{and}\quad
		\epsilon_x=
		\begin{cases}
			1 & \text{if } x\in T, \\
			0 & \text{otherwise},
		\end{cases}
	\]
	then the contribution of $\eb_x+\eb_y$ is $\epsilon_y-\epsilon_x$.
	Hence
	\[
		\sum_{v\in N_G(T)}(\one+\eb_x+\eb_y)_v
		-
		\sum_{u\in T}(\one+\eb_x+\eb_y)_u
		=
		\delta_G(T)+\epsilon_y-\epsilon_x.
	\]
	By \Cref{lem:strict-hall}, $\delta_G(T)\ge 1$.
	Hence the above quantity is non-positive if and only if
	\[
		\delta_G(T)=1,\qquad x\in T,\qquad y\notin N_G(T).
	\]
	Thus (ii) holds.

	Conversely, assume (ii).
	For each non-edge $(x,y)\in\calN_G$, choose a tight acceptable set $T\subset X$ such that $x\in T$ and $y\notin N_G(T)$.
	The same computation gives
	\[
		\sum_{v\in N_G(T)}(\one+\eb_x+\eb_y)_v
		-
		\sum_{u\in T}(\one+\eb_x+\eb_y)_u
		=
		0.
	\]
	Hence $\one+\eb_x+\eb_y$ lies on the corresponding acceptable facet and is not in $\int((n+1)P_G)$.
	Therefore the set counted in \Cref{prop:difference} is empty, and $h_1=h_{s-1}$.
	This proves the equivalence of (i) and (ii).
	The same argument, using the $Y$-side form of the Ohsugi--Hibi facet description, gives the direct equivalence of (i) and (iii).
\end{proof}

\begin{rem}\label{rem:graph-separation}
	The equivalence of (ii) and (iii) in \Cref{prop:nonedge-cut} also has a direct graph-theoretic explanation.
	If $T\subset X$ is tight and acceptable, then $S=Y\setminus N_G(T)$ satisfies $N_G(S)=X\setminus T$.
	The two connected graphs in the definition of acceptability are interchanged, so $S$ is acceptable.
	Since $\card{X}=\card{Y}$, we have $\card{N_G(S)}-\card{S}=\card{N_G(T)}-\card{T}=1$, so $S$ is tight.
	Thus $T$ and $S$ separate the same non-edges, and the reverse implication follows by exchanging $X$ and $Y$.
	This correspondence allows the proof of \Cref{prop:graph} to be carried out entirely within graph theory, without reference to the $h$-vector.
\end{rem}

\section{A graph-theoretic criterion}
\label{sec:graph}

In this section we prove the graph-theoretic implication needed in the $2$-connected case of the main theorem.
The key point is that, in the graph-theoretic setting, the $X$-side tight acceptable separation condition in \Cref{prop:nonedge-cut} is equivalent to the tightness of every acceptable set.
Throughout this section, the statements are formulated for $2$-connected matching-covered bipartite graphs, the situation supplied by the pseudo-Gorenstein criterion in the $2$-connected case.
We first record a simple uncrossing lemma for tight subsets.

\begin{lem}[Uncrossing]\label{lem:uncrossing}
	Let $G$ be a $2$-connected matching-covered bipartite graph with bipartition $X\sqcup Y$.
	Suppose that $U,V\subset X$ are tight, $U\cup V\subsetneq X$, and
	\[
		N_G(U)\cap N_G(V)\ne\emptyset.
	\]
	Then $U\cup V$ is tight.
	The same assertion holds after exchanging $X$ and $Y$.
\end{lem}

\begin{proof}
	Since $U$ and $V$ are tight, we have $\delta_G(U)=\delta_G(V)=1$.
	Since $N_G(U\cup V)=N_G(U)\cup N_G(V)$, inclusion--exclusion gives
	\[
		\delta_G(U\cup V)
		=
		\delta_G(U)+\delta_G(V)-\rbra{\card{N_G(U)\cap N_G(V)}-\card{U\cap V}}.
	\]
	If $U\cap V=\emptyset$, then the hypothesis gives
	\[
		\card{N_G(U)\cap N_G(V)}-\card{U\cap V}\ge 1.
	\]
	If $U\cap V\ne\emptyset$, then $N_G(U\cap V)\subset N_G(U)\cap N_G(V)$, and \Cref{lem:strict-hall} gives
	\[
		\card{N_G(U)\cap N_G(V)}-\card{U\cap V}
		\ge
		\card{N_G(U\cap V)}-\card{U\cap V}
		\ge 1.
	\]
	Therefore $\delta_G(U\cup V)\le 1$.
	On the other hand, $U\cup V$ is a proper subset of $X$, and \Cref{lem:strict-hall} gives $\delta_G(U\cup V)\ge 1$.
	Thus $\delta_G(U\cup V)=1$, as desired.
	The proof of the symmetric assertion is identical.
\end{proof}

The next lemma isolates the local construction that produces tight subsets inside an acceptable set.

\begin{lem}[Internal tight cover]\label{lem:internal-tight-cover}
	Let $G$ be a $2$-connected matching-covered bipartite graph with bipartition $X\sqcup Y$.
	Assume that for every non-edge $(x,y)\in X\times Y$ there exists a tight subset $S\subset Y$ such that
	\[
		y\in S,\qquad x\notin N_G(S).
	\]
	If $A\subset X$ is an acceptable set, then, for every $x\in A$, there exists a tight subset $R_x\subset A$ of $X$ such that $x\in R_x$.
\end{lem}

\begin{proof}
	Put
	\[
		D=N_G(A),\qquad B=X\setminus A,\qquad C=Y\setminus D.
	\]
	Since $A$ is acceptable, the induced graphs on $A\sqcup D$ and $B\sqcup C$ are connected, and the latter graph has at least one edge.
	There is no edge between $A$ and $C$.
	Since the graph on $B\sqcup C$ is connected, this gives $N_G(C)=B$.
	Thus $C$ is an acceptable subset of $Y$.

	Put
	\[
		r=\delta_G(C)=\card{B}-\card{C}.
	\]
	Since $\card{X}=\card{Y}$, we have
	\[
		\card{D}
		=
		\card{Y}-\card{C}
		=
		\card{X}-\card{B}+r
		=
		\card{A}+r.
	\]
	For each $x\in A$, we construct such a set $R_x$ as follows.
	For every $y\in C$, the pair $(x,y)$ is a non-edge.
	Hence there is a tight set $S_y\subset Y$ such that $y\in S_y$ and $x\notin N_G(S_y)$.

	The graph on $B\sqcup C$ is connected.
	Moreover, all intermediate unions avoid $x$ in their neighborhoods, and hence are proper subsets of $Y$.
	Repeated use of the symmetric form of \Cref{lem:uncrossing} yields a tight set $U_x\subset Y$ such that
	\[
		C\subset U_x,\qquad x\notin N_G(U_x).
	\]
	Put
	\[
		P_x=U_x\setminus C\subset D,\qquad
		Q_x=N_G(P_x)\cap A.
	\]
	Then $N_G(U_x)=B\sqcup Q_x$, and the tightness of $U_x$ gives
	\[
		1=\card{B}+\card{Q_x}-\card{C}-\card{P_x}.
	\]
	Hence
	\[
		\card{P_x}-\card{Q_x}=r-1.
	\]

	Let $R_x=A\setminus Q_x$.
	Since $x\notin N_G(U_x)$, we have $x\in R_x$.
	Also, no vertex of $C$ is adjacent to $R_x\subset A$, and no vertex of $P_x$ is adjacent to $R_x$ by the definition of $Q_x$.
	Therefore $N_G(R_x)\subset D\setminus P_x$.
	Moreover,
	\[
		\card{D\setminus P_x}
		=
		\card{D}-\card{P_x}
		=
		\card{A}+r-\rbra{\card{Q_x}+r-1}
		=
		\card{R_x}+1.
	\]
	Since $x\in R_x\subset A$ and $A$ is acceptable, $R_x$ is a non-empty proper subset of $X$.
	Hence \Cref{lem:strict-hall} gives
	\[
		\card{N_G(R_x)}\ge \card{R_x}+1.
	\]
	Thus
	\[
		N_G(R_x)=D\setminus P_x,\qquad \card{N_G(R_x)}=\card{R_x}+1.
	\]
	Hence $R_x$ is tight.
\end{proof}

\begin{prop}\label{prop:graph}
	Let $G$ be a $2$-connected matching-covered bipartite graph with bipartition $X\sqcup Y$.
	The following conditions are equivalent:
	\begin{enumerate}
		\item for every non-edge $(x,y)\in X\times Y$, there exists a tight acceptable set $T\subset X$ such that
		      \[
			      x\in T,\quad y\notin N_G(T);
		      \]
		\item every acceptable set is tight.
	\end{enumerate}
\end{prop}

\begin{proof}
	Assume (i).
	By the graph-theoretic correspondence in \Cref{rem:graph-separation}, for every non-edge $(x,y)\in X\times Y$ there exists a tight acceptable set $S\subset Y$ such that $y\in S$ and $x\notin N_G(S)$.
	In particular, the hypothesis of \Cref{lem:internal-tight-cover} is satisfied.

	Let $A\subset X$ be an acceptable set.
	By \Cref{lem:internal-tight-cover}, every $x\in A$ belongs to a tight subset $R_x\subset A$.
	Since the graph induced by the edges between $A$ and $N_G(A)$ is connected, the sets $R_x$ cover $A$, and all intermediate unions remain contained in the proper subset $A\subsetneq X$, repeated use of \Cref{lem:uncrossing} shows that their union, namely $A$, is tight.
	Hence every acceptable set in $X$ is tight.

	For acceptable subsets of $Y$, the symmetric argument uses the original $X$-side condition in (i) and gives the same conclusion.
	Thus (ii) holds.

	Conversely, assume (ii), and let $(x,y)\in X\times Y$ be a non-edge.
	Since $G$ is $2$-connected, $G\setminus y$ is connected.
	Let $H$ be the induced subgraph on
	\[
		(X\setminus N_G(y))\sqcup (Y\setminus\set{y}),
	\]
	and let $T$ be the set of vertices in $X$ lying in the connected component of $H$ that contains $x$.
	Then $x\in T$ and $y\notin N_G(T)$.
	Moreover, the graph induced by the edges between $T$ and $N_G(T)$ is connected by construction.
	Since $G\setminus y$ is connected, every other connected component of $H$ is adjacent to a vertex of $N_G(y)$.
	It follows that the graph induced by the edges between $X\setminus T$ and $Y\setminus N_G(T)$ is connected; it contains the edges from $y$ to $N_G(y)$ and all other components of $H$ attach to these vertices.
	This graph has at least one edge.
	Hence $T$ is acceptable.
	By (ii), the set $T$ is tight, proving (i).
\end{proof}

\section{Proof of the main theorem}

We now combine the coefficient interpretation from \Cref{sec:coefficient} with the graph-theoretic criterion from \Cref{sec:graph}.
This completes the passage from the numerical condition $h_1=h_{s-1}$ to the Ohsugi--Hibi Gorenstein criterion in the $2$-connected case.
The block decomposition then gives the stated result for arbitrary bipartite graphs.

We use the following form of the Ohsugi--Hibi criterion, rewritten in the terminology of acceptable sets.
In their notation the statement is formulated for the toric ring of $P_G$, which is naturally isomorphic to our edge ring $\kk[G]$ as a standard graded algebra.

\begin{thm}[{Ohsugi--Hibi Gorenstein Criterion \cite[Theorem~2.1(a$'$)]{ohsugi2006GorEdge}}]\label{thm:ohcriterion}
	Let $G$ be a $2$-connected bipartite graph with bipartition $X\sqcup Y$.
	Then $\kk[G]$ is Gorenstein if and only if $G$ has a perfect matching and
	\[
		\card{N_G(T)}=\card{T}+1
	\]
	for every acceptable set $T\subset X$.
\end{thm}

\begin{prop}\label{prop:two-connected-main}
	Let $G$ be a $2$-connected bipartite graph and write
	\[
		h(\kk[G];t)=h_0+h_1t+\cdots+h_st^s.
	\]
	If $h_s=1$ and $h_1=h_{s-1}$, then $\kk[G]$ is Gorenstein.
\end{prop}

\begin{proof}
	Since $G$ is $2$-connected and bipartite, the assumption $h_s=1$ implies that $G$ is matching-covered by the pseudo-Gorenstein criterion for bipartite edge rings~\cite{hatasa2025pseudo}.
	In this setting, \Cref{prop:nonedge-cut,prop:graph} show that $h_1=h_{s-1}$ is equivalent to the tightness of every acceptable set.
	Hence every acceptable set is tight.
	Since a matching-covered graph has a perfect matching, \Cref{thm:ohcriterion} implies that $\kk[G]$ is Gorenstein.
\end{proof}

\begin{proof}[Proof of \Cref{thm:main}]
	Isolated vertices do not appear in the generators of $\kk[G]$, so they may be ignored.
	If $G$ has no edges, then $\kk[G]=\kk$ and the assertion is clear.
	Let $B_1,\ldots,B_m$ be the blocks of $G$, and write
	\[
		h(\kk[B_i];t)=h_{i,0}+h_{i,1}t+\cdots+h_{i,s_i}t^{s_i}.
	\]
	By \Cref{prop:block-product},
	\[
		h(\kk[G];t)=\prod_{i=1}^m h(\kk[B_i];t).
	\]
	Hence
	\[
		h_s=\prod_{i=1}^m h_{i,s_i}.
	\]
	Since the block edge rings are Cohen--Macaulay, their $h$-polynomials have non-negative integer coefficients.
	Thus $h_s=1$ implies $h_{i,s_i}=1$ for every $i$.
	If $s=0$, then $h(\kk[G];t)=1$, so $\kk[G]$ is Gorenstein by Stanley's theorem~\cite{stanley1978hilbert}.
	Thus we may assume $s>0$.

	Blocks with $s_i=0$ have $h(\kk[B_i];t)=1$, and hence their edge rings are Gorenstein by Stanley's theorem.
	For any block with $s_i>0$, the graph $B_i$ is $2$-connected and bipartite, and $\kk[B_i]$ is pseudo-Gorenstein.
	By the pseudo-Gorenstein criterion for bipartite edge rings~\cite{hatasa2025pseudo}, this block is matching-covered.
	Applying \Cref{prop:difference} to $B_i$ gives
	\[
		h_{i,s_i-1}-h_{i,1}\ge 0.
	\]
	On the other hand, since $h_{i,0}=1$ and $h_{i,s_i}=1$ for every $i$, comparing the coefficients of $t$ and of the next-to-leading term in the product formula gives
	\[
		h_1=\sum_{s_i>0} h_{i,1},
		\qquad
		h_{s-1}=\sum_{s_i>0} h_{i,s_i-1}.
	\]
	Since $h_1=h_{s-1}$, all the non-negative differences
	\[
		h_{i,s_i-1}-h_{i,1}
	\]
	are zero.
	Thus, for each block with $s_i>0$, \Cref{prop:two-connected-main} shows that $\kk[B_i]$ is Gorenstein.

	Therefore every block has a palindromic $h$-polynomial.
	Their product $h(\kk[G];t)$ is palindromic, and Stanley's theorem again implies that $\kk[G]$ is Gorenstein.
\end{proof}

The same argument gives a Gorenstein closure obtained by adjoining all non-edges that are not separated by tight acceptable sets.

\begin{prop}\label{prop:tight-separation-closure}
	Let $G$ be a $2$-connected matching-covered bipartite graph with bipartition $X\sqcup Y$.
	Write
	\[
		h(\kk[G];t)=h_0+h_1t+\cdots+h_st^s.
	\]
	Let $\Fill(G)$ be the set of non-edges $\{x,y\}$, with $x\in X$ and $y\in Y$, that are not separated by any tight acceptable set $T\subset X$, in the sense that no such $T$ satisfies
	\[
		x\in T,\qquad y\notin N_G(T).
	\]
	Let $\widehat{G}$ be the bipartite graph on the same bipartition with
	\[
		E(\widehat{G})
		=
		E(G)\cup\Fill(G).
	\]
	Then the following assertions hold.
	\begin{enumerate}
		\item $\widehat{G}$ is again $2$-connected and matching-covered, and $\kk[\widehat{G}]$ is Gorenstein.
		\item Writing
		      \[
			      h(\kk[\widehat{G}];t)=\widehat{h}_0+\widehat{h}_1t+\cdots+\widehat{h}_{\widehat{s}}t^{\widehat{s}},
		      \]
		      one has $\widehat{s}=s$ and $\widehat{h}_{s-1}=h_{s-1}$.
		\item $\widehat{G}$ is minimal among Gorenstein bipartite supergraphs with this bipartition: if $K$ is a bipartite graph with bipartition $X\sqcup Y$ such that $E(G)\subset E(K)$ and $\kk[K]$ is Gorenstein, then
		      \[
			      E(\widehat{G})\subset E(K).
		      \]
	\end{enumerate}
\end{prop}

\begin{proof}
	We first prove (i).
	Adding edges preserves $2$-connectedness, so $\widehat{G}$ is $2$-connected.
	By \Cref{lem:strict-hall}, $\card{X}=\card{Y}$.
	We show that $\widehat{G}$ is matching-covered.
	Every edge of $G$ is contained in a perfect matching of $G$, and hence also in a perfect matching of $\widehat{G}$.
	Let $\{x,y\}\in\Fill(G)$ be a new edge of $\widehat{G}$.
	For every non-empty subset $A\subset X\setminus\set{x}$, \Cref{lem:strict-hall} gives
	\[
		\card{N_G(A)}\ge \card{A}+1,
	\]
	and therefore
	\[
		\card{N_G(A)\setminus\set{y}}\ge \card{A}.
	\]
	Since $\card{X\setminus\set{x}}=\card{Y\setminus\set{y}}$, Hall's theorem gives a perfect matching of the induced graph $G\setminus\set{x,y}$.
	Together with the edge $\{x,y\}$, this gives a perfect matching of $\widehat{G}$ containing $\{x,y\}$.

	We next verify the $X$-side condition in \Cref{prop:graph} for $\widehat{G}$.
	Let $(x,y)$ be a non-edge of $\widehat{G}$.
	Then $\{x,y\}\notin\Fill(G)$, so there exists a tight acceptable set $T\subset X$ for $G$ such that
	\[
		x\in T,\qquad y\notin N_G(T).
	\]

	We claim that no added edge of $\widehat{G}$ joins a vertex of $T$ to a vertex of $Y\setminus N_G(T)$.
	Indeed, if $a\in T$ and $b\notin N_G(T)$, then the same set $T$ separates the non-edge $\{a,b\}$, so $\{a,b\}\notin\Fill(G)$.
	It follows that
	\[
		N_{\widehat{G}}(T)=N_G(T).
	\]
	Consequently, $T$ remains tight for $\widehat{G}$.
	The two connected graphs in the acceptable condition for $T$ can only gain edges when passing from $G$ to $\widehat{G}$, so $T$ is also acceptable for $\widehat{G}$.

	Thus every non-edge of $\widehat{G}$ is separated by a tight acceptable subset of $X$.
	By \Cref{prop:graph}, every acceptable set of $\widehat{G}$ is tight.
	Since $\widehat{G}$ has a perfect matching, \Cref{thm:ohcriterion} implies that $\kk[\widehat{G}]$ is Gorenstein.

	We next prove (ii).
	The argument in the proof of \Cref{prop:nonedge-cut} identifies $\Fill(G)$ with the set counted in \Cref{prop:difference}, so
	\[
		\card{\Fill(G)}=h_{s-1}-h_1.
	\]
	By \Cref{lem:h1}, the first coefficient of $h(\kk[\widehat{G}];t)$ is
	\[
		\widehat{h}_1=\card{E(\widehat{G})}-\card{V(G)}+1=h_1+\card{\Fill(G)}=h_{s-1}.
	\]
	Since $\widehat{G}$ is again $2$-connected and matching-covered with the same bipartition as $G$, its $h$-polynomial has degree $\widehat{s}=s$.
	Since $\kk[\widehat{G}]$ is Gorenstein, its $h$-polynomial is palindromic; hence $\widehat{h}_s=1$ and $\widehat{h}_{s-1}=\widehat{h}_1=h_{s-1}$.

	Finally, we prove (iii).
	Let $K$ be a bipartite graph with bipartition $X\sqcup Y$ such that $E(G)\subset E(K)$ and $\kk[K]$ is Gorenstein.
	As $G$ is $2$-connected, the supergraph $K$ is also $2$-connected.
	Since $\kk[K]$ is Gorenstein, it is pseudo-Gorenstein; by the pseudo-Gorenstein criterion for bipartite edge rings~\cite{hatasa2025pseudo}, $K$ is matching-covered.
	The $h$-polynomial of $\kk[K]$ is palindromic, so its first and next-to-leading coefficients are equal.
	Let $\{x,y\}$ be a non-edge of $K$, with $x\in X$ and $y\in Y$.
	Then $\{x,y\}$ is also a non-edge of $G$.
	Applying \Cref{prop:nonedge-cut} to the non-edge $(x,y)$ of $K$, there exists a tight acceptable set $T\subset X$ for $K$ such that
	\[
		x\in T,\qquad y\notin N_K(T).
	\]

	We show that $T$ is also a tight acceptable set for $G$.
	Since $E(G)\subset E(K)$, we have $N_G(T)\subset N_K(T)$.
	The tightness of $T$ in $K$ gives $\card{N_K(T)}=\card{T}+1$, while \Cref{lem:strict-hall} applied to $G$ gives
	\[
		\card{N_G(T)}\ge\card{T}+1.
	\]
	Thus $N_G(T)=N_K(T)$.
	In particular, $T$ is tight for $G$ and $y\notin N_G(T)$.

	It remains to check acceptability in $G$.
	Put $D=N_G(T)$.
	If the graph induced by the edges of $G$ between $T$ and $D$ had $r$ connected components, and if $T_1,\ldots,T_r$ were their $X$-parts, then each $T_i$ would be a non-empty proper subset of $X$.
	Hence \Cref{lem:strict-hall} would give
	\[
		\card{D}
		=
		\sum_{i=1}^r \card{N_G(T_i)}
		\ge
		\sum_{i=1}^r \rbra{\card{T_i}+1}
		=
		\card{T}+r.
	\]
	Since $\card{D}=\card{T}+1$, we get $r=1$.
	Hence the graph induced by the edges between $T$ and $D$ is connected.

	Next put $C=Y\setminus D$.
	The set $C$ is non-empty because $y\in C$.
	Since $T$ is non-empty and $G$ is connected, $D$ is non-empty, so $C$ is a proper subset of $Y$.
	Moreover, no vertex of $C$ is adjacent to a vertex of $T$, so $N_G(C)\subset X\setminus T$.
	The symmetric form of \Cref{lem:strict-hall}, together with the equality
	\[
		\card{X\setminus T}
		=
		\card{X}-\card{T}
		=
		\card{Y}-\card{D}+1
		=
		\card{C}+1,
	\]
	gives $N_G(C)=X\setminus T$.
	The same component-count argument, now applied to the subset $C\subset Y$, shows that the graph induced by the edges between $X\setminus T$ and $C$ is connected.
	Thus the complementary graph in the acceptable condition is connected and has an edge.
	Therefore $T$ is acceptable for $G$.

	This tight acceptable set satisfies $x\in T$ and $y\notin N_G(T)$.
	Hence $\{x,y\}\notin\Fill(G)$.
	Therefore every non-edge of $K$ is not in $\Fill(G)$, and so every edge in $\Fill(G)$ belongs to $E(K)$.
	Thus $E(\widehat{G})\subset E(K)$.
\end{proof}

\begin{proof}[Proof of \Cref{thm:blockwise-closure}]
	Isolated vertices do not affect the edge ring, so we ignore them.
	If $G$ has no edges, then $\widehat{G}=G$ and $\kk[G]=\kk$, so the assertion is clear.
	Hence assume that $G$ has at least one edge.
	Let $B_1,\ldots,B_m$ be the blocks of $G$ with at least one edge.
	Write
	\[
		h(\kk[B_i];t)=h_{i,0}+h_{i,1}t+\cdots+h_{i,s_i}t^{s_i}.
	\]
	By \Cref{prop:block-product},
	\[
		h(\kk[G];t)=\prod_{i=1}^m h(\kk[B_i];t).
	\]
	Since each block edge ring is Cohen--Macaulay, all coefficients of these factors are non-negative integers.
	Thus $h_s=1$ implies that the leading coefficient of each factor is one.
	Thus every block $B_i$ is pseudo-Gorenstein.

	Let $\widehat{B}_i$ be the graph obtained from $B_i$ by adjoining the edges in $\Fill(B_i)$, and write
	\[
		h(\kk[\widehat{B}_i];t)=\widehat{h}_{i,0}+\widehat{h}_{i,1}t+\cdots+\widehat{h}_{i,\widehat{s}_i}t^{\widehat{s}_i}.
	\]
	If $B_i$ is a single-edge block, then $\Fill(B_i)=\emptyset$, $\widehat{B}_i=B_i$, $\kk[B_i]$ is a polynomial ring, and $\widehat{s}_i=s_i=0$.
	Otherwise, $B_i$ is $2$-connected and bipartite.
	By the pseudo-Gorenstein criterion for bipartite edge rings~\cite{hatasa2025pseudo}, the block $B_i$ is matching-covered.
	Hence \Cref{prop:tight-separation-closure} applies to it, so $\widehat{B}_i$ is $2$-connected and matching-covered and $\kk[\widehat{B}_i]$ is Gorenstein; moreover, $\widehat{h}_{i,\widehat{s}_i-1}=h_{i,s_i-1}$.
	Since $B_i$ and $\widehat{B}_i$ have the same bipartition, their $h$-polynomials have the same degree; hence $\widehat{s}_i=s_i$.

	We first prove (i).
	All added edges lie inside the original blocks, so the blocks of $\widehat{G}$ are precisely these filled blocks together with the unchanged single-edge blocks.
	Therefore, by \Cref{prop:block-product},
	\[
		h(\kk[\widehat{G}];t)=\prod_{i=1}^m h(\kk[\widehat{B}_i];t).
	\]
	Each factor is palindromic; hence this product is palindromic, and its degree is $\widehat{s}=\sum_i\widehat{s}_i=\sum_i s_i=s$.
	In particular, $\widehat{h}_{\widehat{s}}=1$.
	Since $\widehat{G}$ is bipartite, $\kk[\widehat{G}]$ is a Cohen--Macaulay domain.
	Stanley's theorem~\cite{stanley1978hilbert} implies that $\kk[\widehat{G}]$ is Gorenstein.

	We next prove (ii).
	Comparison of next-to-leading coefficients in the two product formulas gives
	\[
		\widehat{h}_{s-1}
		=
		\sum_{s_i>0}\widehat{h}_{i,\widehat{s}_i-1}
		=
		\sum_{s_i>0}h_{i,s_i-1}
		=
		h_{s-1}.
	\]

	Finally, we prove (iii).
	The asserted blockwise minimality is trivial for single-edge blocks and follows from the minimality assertion in \Cref{prop:tight-separation-closure} for all other blocks.
\end{proof}

\begin{rem}
	The argument shows that possible counterexamples to palindromicity with $h_s=1$ and $h_1=h_{s-1}$ cannot occur among bipartite graphs.
	Without the bipartite assumption, such counterexamples do exist.
	For instance, as observed in~\cite{hatasa2025pseudo}, the Petersen graph has
	\[
		h(\kk[G];t)=1+5t+15t^2+25t^3+5t^4+t^5,
	\]
	so $h_s=1$ and $h_1=h_{s-1}$, but the $h$-polynomial is not palindromic.
\end{rem}

\subsection*{Acknowledgments}
The author would like to thank Nobukazu Kowaki and Koji Matsushita for their previous collaboration on pseudo-Gorenstein edge rings.
The author is also grateful to Tam\'as K\'alm\'an for a discussion on Gorenstein closures.

\bibliographystyle{plain}
\bibliography{References}

\end{document}